\documentclass[reqno,12pt]{amsart}
\usepackage[utf8]{inputenc}
\usepackage[left=1.2in, right=1.2in, top=1in]{geometry}

\usepackage{amsmath, amssymb,amsthm, mathrsfs,color,times,textcomp,verbatim}
\usepackage{xcolor}
\usepackage[colorlinks=true]{hyperref}
\hypersetup{urlcolor=blue, citecolor=red, linkcolor=blue}
\usepackage[square,numbers]{natbib}
\usepackage{pgfplots}
\usepackage{tikz}
\numberwithin{equation}{section}
\theoremstyle{plain}
\newtheorem{theorem}{Theorem}[section]
\newtheorem{proposition}[theorem]{Proposition}

\newtheorem{lemma}[theorem]{Lemma}
\newtheorem{conjecture}{Conjecture}

\allowdisplaybreaks[4]\usepackage{graphicx}

\theoremstyle{definition}

\title{Best Constant and extremal functions for a class  HARDY-SOBOLEV-MAZ'YA INEQUALITIES}
\author{Daowen Lin}
\address{Geometric Partial Differential Equations Unit, Okinawa Institute of Science and Technology Graduate University, 1919-1 Tancha, Onna-son, Kunigami-gun, Okinawa, 904-0495, Japan}
\email{daowen.lin@oist.jp}
\author{Xi-Nan Ma}
\address{School of Mathematical Sciences, University of Science and Technology of China, Hefei, Anhui Province, China, 230026}
\email{xinan@ustc.edu.cn}

\begin{document}
\maketitle
\begin{abstract}
	We derive an integral identity  for a class $p$-Laplace equation, and then classify all positive finite energy cylindrically symmetric  solutions of the equation (\ref{1.2}) for  $3\leq k\leq n-1,$ with the help of some a prior estimates. Combining this with the result of Secchi-Smets-Willem{\cite{SSW03}}, as a consequence, we obtain the best constant and extremal functions  for  the related Hardy-Sobolev-Maz'ya inequalities.
\end{abstract}
\section{Introduction}

In this paper, we study the positive extremals for the following Hardy-Sobolev-Maz'ya inequalities when $s=1, k\ge3$.
\begin{theorem}{\cite{BT02}}
	Let $n\geq 3$, $2\leq k\leq n$, and $p,s$ be real numbers satisfying $1<p<n$, $0\leq s\leq p$ and $s<k$, There exists a positive constant $S_{p,s}=S(p,s,n,k)$ such that for all $u\in D^{1,p}(\mathbb{R}^n)$, we have
	\begin{align}\label{1.1}
		\left(\int_{\mathbb{R}^{n}}\frac{|u|^{\frac{p(n-s)}{n-p}}}{|y|^{s}}dx\right)^{\frac{n-p}{p(n-s)}}\leq S_{p,s}\left(\int_{\mathbb{R}^n}|\nabla u|^pdx\right)^{\frac{1}{p}},
	\end{align}
	where for $p>1$, we define the space $D^{1,p}(\mathbb{R}^n)$ as the closure of $C_c^{\infty}(\mathbb{R}^n)$ with respect to the norm 
	\begin{align*}
		||u||_{D^{1,p}(\mathbb{R}^n)}=\left(\int_{\mathbb{R}^n}|\nabla u|^pdx\right)^{\frac{1}{p}}.
	\end{align*} 
\end{theorem}
For $s=1$, the Euler-Lagrange equation associated to (\ref{1.1}) is
\begin{equation}\label{1.2}
\begin{aligned}
	-\Delta_p u&=\frac{u^{p^*(1)-1}}{|y|}\,\, \textrm{in}\,\, \mathbb{R}^n,\\
	u&>0,\\
	u&\in D^{1,p}(\mathbb{R}^n),
\end{aligned}
\end{equation}
where $p^*(1)=\frac{p(n-1)}{n-p}$, $x=(y,z)$, $\mathbb{R}^n=\mathbb{R}^{k}\times\mathbb{R}^{n-k}$.

 If $s=p$, (\ref{1.1}) is an extension of the classical Hardy inequalities.
 
 The case $s=0$ corresponds to the classical Sobolev inequalities, whose finite energy extremal functions have been classified by \cite{Sci16} and \cite{Vet16}, or see \cite{CFR20} for the convex cone case via integral by part method and a priori estimates. 
 
In more general case $0\leq s<p$ and $k=n$, the best constant have been computed in \cite{GY 2000}, their extremal functions are, up to dilations and translations,
\begin{align*}
	U(x)=(1+|x|^{\frac{p-s}{p-1}})^{-\frac{n-p}{p-s}}
\end{align*} 

More general inequalities have been considered by Caffarelli-Kohn-Nirenberg\cite{CKN}:
\begin{align*}
	\left(\int_{\mathbb{R}^d}|x|^{-bq}u^qdx\right)^{\frac{1}{q}}\leq C_{a,b}\left(\int_{\mathbb{R}^d}|x|^{-ap}|\nabla u|^{p}dx\right)^{\frac{1}{p}}
\end{align*}
where $p>1, a<b<a+1, q=\frac{dp}{d-p(1+a-b)}, 0\leq a<\frac{d}{p}-1, C_{a,b}=C_{a,b}(d,p,a,b,q).$ They turn out to be of the same form as in the case $a=0$, for general, we refer to \cite{DEM16} and \cite{CC21} for the recent progress. 

When $k<n$, extremal functions cannot be anymore radially symmetric. Cylindrical symmetry of some extremal functions has been proved Secchi-Smets-Willem \cite{SSW03}, by symmetrizations arguments.

For $p=2$, $s=1$, Mancini-Fabbri-Sandeep \cite{MFS06} classified all positive finite energy solutions  using moving plane method and  a differential identity from \cite{GV} with some a prior pointwise estimates. 

Independently, Alvino-Ferone-Trombetti also investigated when $s=1$ in \cite{AFT06}, they have a conjecture as follow in their article, but they could only address when $n=3, k=2, p=2$.
\begin{conjecture}\cite{AFT06}\label{1.3}
	Let $u_0$ be the function given by 
	\begin{align}
		u_0(x)=u_0(y,z)=C_{p,n,k}[(1+|y|)^2+|z|^2]^{-\frac{n-p}{2(p-1)}},
	\end{align}
	where $C_{p,n,k}=(k-1)^{\frac{n-p}{p(p-1)}}(n-p)^{\frac{n-p}{p}}(p-1)^{-\frac{n-p}{p}}$. Then $u$ is an extremal function for (\ref{1.1}) if $u(y,z)=\lambda^{\frac{n-p}{p}}u_0(\lambda y,\lambda z+z_0)$ for some $\lambda>0$ and $z_0\in \mathbb{R}^{n-k}$.
\end{conjecture}
Invoking the cylindrical symmetry results in \cite{SSW03},  we generalize the differential identity in \cite{GV}, which is motivated from the idea by Jerison-Lee \cite{JL1987}. But unfortunately, we can't get the similar a prior point estimates for $p$-Laplace equations. Since what we need to characterize the cylindrically symmetric solutions of (\ref{1.2}) are almost integral estimates, with the help of sharp $C^0$ estimate in \cite{Dut22},  the gradient estimates using the ideas from Xiang \cite{Xiang17} and  Antonini-Ciraolo-Farina\cite{ACF} . We can tackle some cases:

\begin{theorem}\label{them1.4}
	For $s=1,n\geq 4,3\leq k\leq n-1$, $u$ is a cylindrically symmetric solution of (\ref{1.2}) then $u(y,z)=\lambda^{\frac{n-p}{p}}u_0(\lambda y,\lambda z+z_0)$ for some $\lambda>0$ and $z_0\in \mathbb{R}^{n-k}$.
\end{theorem}

From the above theorem, we can give a positive answer for the 
conjecture by Alvino-Ferone-Trombetti\cite{AFT06} when $k\ge 3$. Since we can't get the a priori estimates, the case $k=2$ is still open.

As we had remarked that the best constant of Hardy-Sobolev-Maz'ya inequalities (\ref{1.1}) achieve at some cylindrically symmetric functions  \cite{SSW03}.
	
  We give some more reference for the best constant and extremal function on the Hardy-Sobolev-Maz'ya inequalities on $p=2$, please see the recent progress in [\cite{BFL08}, \cite{CL08},  \cite{CFMS08}, \cite{CFMS09}, \cite{MS08}, \cite{V11}] etc., they used the hyperbolic symmetry or the moving plane methods.
  
  The paper is organized as follows. In section 2, we first review the cylindrically symmetric for the extremal function by Secchi-Smets-Willem \cite{SSW03}. Then we give the  sharp $C^0$ estimate by Dutta \cite{Dut22},  the gradient estimates using the idea from  Xiang \cite{Xiang17} and Antonini-Ciraolo-Farina\cite{ACF}.  In section 3, we  give a generalization of the differential  identity for a class p-Laplace operator  via the ideas from 	Serrin-Zou \cite{SZ02},  \cite{GV} and \cite{MFS06}, as we noted which is motivated from the idea Jerison-Lee \cite{JL1987} for the Laplace operator case.
  Then, using this generalized differential identity and  the integral estimates  in section 2, we complete the proof of Theorem \ref{them1.4}.

\section{cylindrically symmetry properties and  a priori estimates}
In this sections, we first show  briefly that the best constant of Hardy-Sobolev-Maz'ya inequalities achieve at some cylindrically symmetric functions in Secchi-Smets-Willem\cite{SSW03}, and  we  state the $C^0$ estimates by Dutta \cite{Dut22}. Then we mainly obtain the integral gradient estimates in Lemma \ref{lem3.3} and Lemma \ref{lem3.4}, where we use the ideas from  Xiang \cite{Xiang17} and  Antonini-Ciraolo-Farina\cite{ACF}.  These integral estimates will be used in next section to get the proof of the main  theorem.

As in \cite{SSW03}, we consider the minimization problem
\begin{align}
S_{p,s}=S(p,s,n,k)=\inf\left\{\int_{\mathbb{R}^n}|\nabla u|^p|u\in D^{1,p}(\mathbb{R}^n), \int_{\mathbb{R}^n}\frac{|u|^{p^*(1)}}{|y|}=1\right\}. 	
\end{align}
Let $u\in L^1(\mathbb{R}^n)$ be a non-negative function. Let us denote by $u^*(*,z)$ the Schwarz symmetrization of $u(*,z)$ and $u^{**}(y,*)$ the Schwarz symmetrization of $u^*(y,*)$. It is clear that $u^{**}$ depends only on $(|y|,|z|)$. Let us define
\begin{align*}
	D_{**}^{1,p}(\mathbb{R}^n)=\{u\in D^{1,p}(\mathbb{R}^n)|u=u^{**}\},
\end{align*} 
and
\begin{align}
S^{**}_{p,s}=S^{**}(p,s,n,k)=\inf\left\{\int_{\mathbb{R}^n}|\nabla u|^p|u\in D_{**}^{1,p}(\mathbb{R}^n), \int_{\mathbb{R}^n}\frac{|u|^{p^*(1)}}{|y|}=1\right \}.
\end{align}
From \cite{SSW03}, first we have the following property.
\begin{proposition}\label{prop3.0}\cite{SSW03}
	$S_{p,s}=S^{**}_{p,s}$.
\end{proposition}
Then we know the best constant of Hardy-Sobolev-Maz'ya inequalities (\ref{1.1}) achieve at some cylindrically symmetric functions.

Now we state a priori $C^0$ estimates  for the Euler-Lagrange equation  (\ref{1.2}) from Dutta \cite{Dut22}.

\begin{proposition}\label{pro3.1}\cite{Dut22}
	Let $u$ be a solution of (\ref{1.2}), then for any $x\in \mathbb{R}^n$, there are universal constants $c, C$ such that 
	\begin{align}\label{3.1}
		\frac{c}{1+|x|^{\frac{n-p}{p-1}}}\leq u(x)\leq \frac{C}{1+|x|^{\frac{n-p}{p-1}}}.
	\end{align}
	\end{proposition}
We state a a priori gradient estimates for a class $p$-Laplace equation from Xiang \cite{Xiang17}.
\begin{proposition}\label{pro3.2}\cite{Xiang17}
	Let $\Omega$ be a domain in $\mathbb{R}^n$ and $f\in L^{\infty}_{loc}(\Omega)$. Let $w\in W^{1,p}_{loc}(\Omega)$ be a weak solution to equation
	\begin{equation*}
		-\Delta_p w=f 
	\end{equation*}
	in $\Omega$, that is, 
	\begin{equation*}
		\int_{\Omega}|\nabla w|^{p-2}\nabla w\cdot \nabla\varphi=\int_{\Omega}f\varphi,\,\, \forall\,\varphi\in C^{\infty}_0(\Omega). 
	\end{equation*} 
	Then, for any ball $B_{2\delta}(x_0)\subset \Omega$, there holds
	\begin{equation}\label{3.1a}
		\sup_{B_{\frac{\delta}{2}}(x_0)}|\nabla w|\leq C\left(\delta^{-n}\int_{B_{\delta}(x_0)}|\nabla w|^p\right)^{\frac{1}{p}}+C\delta^{\frac{1}{p-1}}||f||^{\frac{1}{p-1}}_{\infty,B_{\delta}(x_0)}.
	\end{equation}
\end{proposition}
 We shall use the Proposition \ref{pro3.1} and Proposition \ref{pro3.2} to get the following integral estimates (\ref{3.3}) when 
$k>2p$.
\begin{lemma}\label{lem3.3}
	Let $u$ be a solution of (\ref{1.2}), for any $z\in \mathbb{R}^{n-k}$, we have 
	\begin{align}\label{3.2}
		|\nabla u(y,z)|\leq \frac{C}{|y|} 
	\end{align}
	as $|y|\rightarrow 0$ .
	
	If $k>2p$, then for any $x\in \mathbb{R}^n$ and any $\gamma \in R$,	 we have the following integral estimates.
	\begin{equation}\label{3.3}
	\begin{aligned}
	\int_{B_R}|\nabla u|^{2p}u^{\gamma}dzdy\leq C(1+R^{n-\frac{2p(n-1)}{p-1}-\frac{\gamma(n-p)}{p-1}}).
		\end{aligned}
	\end{equation}	
	\end{lemma}

		\begin{proof}
	For $K\geq1$ , we define $v(y,z)=K^{\frac{n-p}{p-1}}u(Ky,Kz)$, then $v\in L^{\infty}(\mathbb{R}^n\backslash B_1(0))$ independent of $K$ by Proposition \ref{pro3.2} and $v$ satisfies
	\begin{equation}\label{3.4}
		-\Delta_pv=\frac{v^{\frac{n(p-1)}{n-p}}}{K|y|}\, \,\textrm{in} \,\{1<|y|<4\}\times\{|z|<2\}.
	\end{equation}
	 By standard regularity of $p$-Laplace equation \cite{DiB83} or \cite{Tol84}, we have
\begin{align*}
	|\nabla v|\leq C \,\,\textrm{in} \,\{2<|y|<3\}\times \{|z|<1\},
\end{align*}
where $C$ is independent of $K$, which is 
\begin{align*}
	|\nabla u(Ky, Kz)|\leq C K^{-\frac{n-1}{p-1}}\,\,\textrm{in} \,\{2<|y|<3\}\times \{|z|<1\}.
\end{align*}
And since
\begin{align*}
	2K\leq|(Ky,Kz)|\leq \sqrt{10}K,
\end{align*}
in other word,
\begin{align*}
	|\nabla u(x)|\leq \frac{C}{1+|x|^{\frac{n-1}{p-1}}},
	\end{align*}
	for $|y|>2$.

For $0<|y_0|\leq 2$, we consider (\ref{3.4}) in $\,\{|y|<3\}\times\{1<|z|<4\}.
$ At $x_0=(y_0,z_0) $, to use Proposition \ref{pro3.2} we take $\delta=\frac{|y_0|}{8}$.	
  By scaling the standard cut-off function, we have a smooth function $\eta(x)$ satisfies $\eta=1$ in $B_\delta(x_0)$, $0\leq\eta\leq 1$, $\eta=0$ outside $B_{2\delta}(x_0)$ and $|\nabla \eta|\leq \frac{C}{\delta}$. By testing (\ref{3.4}) with 
  \begin{align*}
  	\psi=v\eta^p.
  \end{align*}
  
We have
\begin{equation*}
	\int_{B_{2\delta}(x_0)}|\nabla v|^p\eta^pdx\leq \int_{B_{2\delta}(x_0)}\frac{v^{\frac{p(n-1)}{n-p}}\eta^p}{|y|}+p\int_{B_{2\delta(x_0)}}|\nabla v|^{p-1}\eta^{p-1}|\nabla \eta|v.
\end{equation*}
By Young's inequality, we arrive at
\begin{align*}
	\int_{B_{2\delta}(x_0)}|\nabla v|^p\eta^pdx\leq C\left[\int_{B_{2\delta}(x_0)}\frac{v^{\frac{p(n-1)}{n-p}}\eta^p}{|y|}+\int_{B_{2\delta}(x_0)}v^{p}|\nabla\eta|^{p}\right]\leq C\delta^{n-p},
	\end{align*}
Then from the estimates (\ref{3.1}) and (\ref{3.1a}), we have
\begin{align*}
	\sup_{B_{\frac{\delta}{2}}(x_0)}|\nabla v|\leq C \delta^{-1},
\end{align*}
which implies 
\begin{align}\label{3.4a}
	|\nabla v(y_0,z_0)|\leq \frac{C}{|y_0|}.
\end{align}
Let $K=1$, which is
\begin{align*}
	|\nabla u(y,z)|\leq \frac{C}{|y|}, \quad\textrm{as}\,|y|\rightarrow 0.
\end{align*}

From (\ref{3.4a}), we get
\begin{align*}
	|\nabla v|^{2p} \leq \frac{1}{|y|^{2p}},\,\,\textrm{if}\,\, k>2p.
\end{align*}
We combine it with the estimates  (\ref{3.1}), for any $\gamma \in R$, and $k>2p$, we have
\begin{align*}
	\int_{B_2\backslash B_1}|\nabla v|^{2p}v^{\gamma}\leq C,
\end{align*}
up to scaling which is 
\begin{align*}
	\int_{B_{2K}\backslash B_{K}}|\nabla u|^{2p}u^{\gamma}dx\leq CK^{n-\frac{2p(n-1)}{p-1}-\gamma\frac{n-p}{p-1}}.
\end{align*}
Let $K=1$, arguing as above in $\,\{|y|<3\}\times\{|z|<4\}.$ we get
\begin{align*}
	\int_{B_1}|\nabla u|^{2p}u^{\gamma}dx\leq C.
\end{align*}
Combining these, it follows that
\begin{align}\label{3.4b}
	\int_{B_{R}}|\nabla u|^{2p}u^{\gamma}dx\leq C(1+R^{n-\frac{2p(n-1)}{p-1}-\gamma\frac{n-p}{p-1}}).
\end{align}
\end{proof}

For the other case $p\geq \frac{k}{2}\geq\frac{3}{2}$, we shall use the technique from   Antonini-Ciraolo-Farina\cite{ACF}    to get the integral gradient estimates.
\begin{lemma}\label{lem3.4}
Let $u$ be a solution of (\ref{1.2}), 	for $p\geq \frac{k}{2}\geq\frac{3}{2}$, we have $|\nabla u|^q\in L^1_{loc}$, and  for any $q>1, \gamma\in R$, we have the following estimates.
	\begin{align}\label{3.4c}
		\int_{B_R}|\nabla u|^q u^\gamma dx\leq C(1+R^{n-\frac{q(n-1)}{p-1}-\gamma\frac{n-p}{p-1}}).
	\end{align}
\end{lemma}
\begin{proof}
	We show it by iteration. As Lemma 3.1 in \cite{ACF}, we
	consider the following equation.
	\begin{equation}
	\begin{aligned}\label{3.4d}
		div((\varepsilon+|\nabla u_\varepsilon|^2)^\frac{p-2}{2}\nabla u_{\varepsilon})&=f_\varepsilon \,\,\textrm{in}\, B_1,\\	u_\varepsilon  &=u\,\,\textrm{on}\, \partial B_1,
		\end{aligned}
	\end{equation}
	where the  boundary condition is to be intended as 
	\begin{align*}
		u_\varepsilon -u\in W^{1,p}_0(B_1).
			\end{align*}
	Here $f=\frac{u^{\frac{p(n-1)}{n-p}}}{|y|}\in L^{k-\nu}(B_1)$ for any $k>\nu>0$. Let $f_\varepsilon=\min\{f,\varepsilon^{-1}\}$, $f_\varepsilon \in L^\infty(B_1), |f_\varepsilon|\leq |f|$. Then $f_\varepsilon \rightarrow f $ in $L^\sigma(B_1)$, where
	\begin{align*}
		\sigma &=2, \textrm{if}\,\, p\geq \frac{2n}{n+2},\\\sigma &=\frac{np}{np-n+p}, \textrm {if}\,\, \frac{k}{2}\leq p<\frac{2n}{n+2},
	\end{align*}
	since $\frac{np}{np-n+p}<\frac{np}{np-n}=\frac{p}{p-1}\leq 3.$ 
	
	We apply [\cite{ACF} Lemma 3.1] to our case, we have for $\varepsilon\in [0,1)$, there is a unique solution in $W^{1,p}_u(B_1)= u+W^{1,p}_0(B_1)$ and 
	\begin{align*}
		u_\varepsilon\rightarrow u \,\,\textrm{strongly in}\, W^{1,p}(B_1), 
	\end{align*}
	up to a subsequence, we have $u_\varepsilon\rightarrow u, \nabla u_{\varepsilon}\rightarrow \nabla u$ a.e.

	Anagin by [\cite{ACF}
	Lemma 3.2], we also have
	\begin{align}\label{3.7}
		u_\varepsilon\in H^2_{loc}(B_1)\cap C^1(B_1),
	\end{align}
	and 
	\begin{align}\label{3.8}
		|\nabla u_\varepsilon|^{p-2}\nabla u_\varepsilon\in H^{1}_{loc}(B_1;\mathbb{R}^n).
	\end{align}
From (\ref{3.4d}), it follows that
	\begin{align*}
		-\partial_i(A^{ij}(\nabla u_\varepsilon)\partial_{jm} u_\varepsilon)=\partial_m f_\varepsilon
	\end{align*}
	where $A^{ij}=(\varepsilon+|\nabla u_\varepsilon|^2)^\frac{p-2}{2}\delta_{ij}+(p-2)(\varepsilon+|\nabla u_\varepsilon|^2)^\frac{p-4}{2}\partial_i u_\varepsilon\partial_j u_\varepsilon$.

The above equation is understood in the sense that for all $\psi\in C_0^\infty(B_1)$, 
\begin{align}\label{3.8a}
	\int_{B_1} A^{ij}(\nabla u_\varepsilon)\partial_{jm}u_{\varepsilon}\partial_i\psi =\int_{B_1} \partial_m f_\varepsilon \psi.
\end{align}

In (\ref{3.8a}), we shall set
$\psi =w^\frac{\alpha}{2}\partial_m u_\varepsilon \eta^2$, where $\eta\in C^\infty_0(B_1)$, $\alpha\geq 0$ and $w=\varepsilon+|\nabla u_\varepsilon|^2$.

From the regularity of $u_\varepsilon$ (\ref{3.7}) and (\ref{3.8}),   the above equality (\ref{3.8a}) also holds for $\psi$. First we have 
 \begin{align}
	\partial_i\psi=w^\frac{\alpha}{2}\partial_{im} u_\varepsilon \eta^2+\frac{\alpha}{2}w^{\frac{\alpha}{2}-1}\partial_i w\partial_m u_{\varepsilon} \eta^2+2 w^\frac{\alpha}{2}\partial_m u_\varepsilon \eta\partial_i \eta.
\end{align}

Since  $|\partial_m f_\varepsilon|\leq |\partial_m f|$,
 and 
\begin{align}
	\partial_m f=\frac{p(n-1)}{n-p}\frac{u^{\frac{n(p-1)}{n-p}}}{|y|}\partial_m u+u^\frac{p(n-1)}{n-p}\partial_m\bigg(\frac{1}{|y|}\bigg).
\end{align}
We know  $u$ is bounded from  (\ref{3.1}), $|\partial_m u|\leq \frac{C}{|y|}$ from  (\ref{3.2}) and $|\partial_m(|y|^{-1})|\leq C|y|^{-2}$.
It follows that 
 \begin{align}
 	|\partial_m f_\varepsilon\psi|\leq C\frac{w^\frac{\alpha+1}{2}}{|y|^2}.
 \end{align}

From [\cite{ACF} Proposition 4.3], for any $\varepsilon\in (0,1)$ and for any open ball $B_{2r}\subset B_1$, we have
\begin{align}\label{3.13}
	\int_{B_r} ||\nabla u_\varepsilon|^{p-2}\nabla u_\varepsilon|^2dx\leq C\bigg[r^{-n}\bigg(\int_{B_{2r}\backslash B_r}|\nabla u_\varepsilon|^{p-1}\bigg)^2+r^2\int_{B_{2r}}f_\varepsilon^2\bigg],
\end{align}
\begin{align}\label{3.14}
	\int_{B_\frac{r}{2}}|\nabla(|\nabla u_\varepsilon|^{p-2}\nabla u_\varepsilon)|^2dx\leq C\bigg[r^{-n-2}\bigg(\int_{B_{2r}\backslash B_r}|\nabla u_\varepsilon|^{p-1}dx\bigg)^2+\int_{B_{2r}}f_\varepsilon^2dx\bigg],
\end{align}
where $C$ is only depending on $n,p$.

Since $k\geq 3$, $f\in L^2_{loc}$ and $\nabla u_\varepsilon\rightarrow\nabla u$ in $L^p$, so  from (\ref{3.13}), (\ref{3.14}) we have

\begin{align*}
	|||\nabla u_\varepsilon |^{p-2}\nabla u_\varepsilon||_{H^{1}_{loc}}\leq C,
\end{align*}
 $C$ is independent of $\varepsilon$.
Therefore, by Hardy-Sobolev-Maz'ya inequalities (\ref{1.1}) with $p=2$ and $s=0$ or $s=2$ to $|\nabla u_\varepsilon |^{p-2}\nabla u_{\varepsilon}$, we get
\begin{align*}
	|\nabla u_\varepsilon|^{p-2}\nabla u_\varepsilon\in L^\frac{2n}{n-2}_{loc}\cap L^\frac{p}{p-1}_{loc} ,\frac{|\nabla u_\varepsilon|^{2p-2}}{|y|^2}\in L^1_{loc},
\end{align*}
then
\begin{align}\label{3.15}
	w\in L^\frac{n(p-1)}{n-2}_{loc}\cap L^\frac{p}{2}_{loc}, \frac{w^{p-1}}{|y|^2}\in L^1_{loc},
\end{align}
and the norms are independent of $\varepsilon.$

Summing up all $m=1,...,n,$ we obtain
\begin{align*}
	&A^{ij}(\nabla u_\varepsilon)\partial_{jm} u_\varepsilon \partial_i\psi\\=& w^\frac{\alpha+p-2}{2}|\partial_{im}u_\varepsilon|^2\eta^2+(p-2)w^\frac{\alpha+p-4}{2}(\partial_{im}u_\varepsilon\partial_{i}u_\varepsilon)^2\eta^2\\&+\frac{\alpha}{2}w^\frac{\alpha+p-4}{2}\partial_i w\partial_m u_\varepsilon\partial_{im} u_{\varepsilon}\eta^2+\frac{\alpha(p-2)}{2}w^\frac{\alpha+p-6}{2}\partial_{jm} u_\varepsilon\partial_i w\partial_m u_\varepsilon\partial_i u_\varepsilon\partial_j u_\varepsilon\eta^2\\&+2w^\frac{\alpha+p-2}{2}\partial_{im}u_\varepsilon\partial_m u_\varepsilon\eta\partial_i\eta+2(p-2)w^\frac{\alpha+p-4}{2}\partial_{jm}u_\varepsilon\partial_m u_\varepsilon\eta \partial_i \eta \partial_i u_\varepsilon\partial_j u_\varepsilon
	\\=&w^\frac{\alpha+p-2}{2}|\nabla^2 u_\varepsilon|^2\eta^2+\frac{\alpha+p-2}{4}w^\frac{\alpha+p-4}{2}|\nabla w|^2\eta^2+\frac{\alpha(p-2)}{4}w^\frac{\alpha+p-6}{4}|\nabla w\cdot\nabla u_\varepsilon|^2\eta^2\\&+2w^\frac{\alpha+p-2}{2}\partial_{im}u_\varepsilon\partial_m u_\varepsilon\eta\partial_i\eta+2(p-2)w^\frac{\alpha+p-4}{2}\partial_{jm}u_\varepsilon\partial_m u_\varepsilon\eta \partial_i \eta \partial_i u_\varepsilon\partial_j u_\varepsilon.
\end{align*}
Since $\partial_i w=2\partial_j u_\varepsilon\partial_{ij} u_\varepsilon$, $|\nabla w|^2=4|\partial_j u_\varepsilon\partial_{ij} u_\varepsilon|^2\leq 4|\nabla u_\varepsilon |^2|\nabla^2 u_\varepsilon |^2\leq 4w|\nabla^2u_\varepsilon|^2$. And  for $\alpha\geq 0, p\geq\frac{3}{2}$, using Cauchy inequality we get

\begin{align*}
	\frac{\alpha}{8}w^{\frac{\alpha+p-4}{2}}|\nabla w|^2\eta^2+\frac{\alpha(p-2)}{4}w^\frac{\alpha+p-4}{2}|\nabla w\cdot \nabla u_\varepsilon|^2\eta^2\geq\frac{\alpha(p-\frac{3}{2})}{4}|\nabla w\cdot \nabla u_\varepsilon|^2\eta^2\geq 0,
	\end{align*}
and 
	\begin{align*}
		\frac{1}{2}w^\frac{\alpha+p-2}{2}|\nabla^2 u_\varepsilon|^2\eta^2+\frac{p-2}{4}w^\frac{\alpha+p-4}{2}|\nabla w|^2\eta^2\geq \frac{p-\frac{3}{2}}{4}w^\frac{\alpha+p-4}{2}|\nabla w|^2\eta^2\geq 0.
			\end{align*}

  It follows that there exist $c_1, c_2>0$ depend on $p$ such that 
\begin{align}
	\Sigma_{m=1}^n A^{ij}(\nabla u_\varepsilon)\partial_{jm}u_\varepsilon\partial_i\psi \geq c_1 w^{\frac{\alpha+p-2}{2}}|\nabla^2 u_\varepsilon|^2\eta^2+c_2w^\frac{\alpha+p-4}{2}|\nabla w|^2\eta^2-Cw^\frac{\alpha+p}{2}|\nabla \eta|^2
\end{align}
Then we get
\begin{align}
	\int_{B_1}|\nabla (\eta w^\frac{\alpha+p}{4})|^2\leq C\bigg(\int_{B_1}w^\frac{\alpha+p}{2}|\nabla\eta|^2+\int_{B_1}\frac{w^\frac{\alpha+1}{2}}{|y|^2}\bigg)
\end{align}

For $\alpha_1=0$, then from (\ref{3.15}), $w^\frac{p}{2}, \frac{w^\frac{1}{2}}{|y|^2}\in L^1(B_{\frac{3}{4}})$, applying (\ref{1.1}) with $p=2, s=2$ to $\eta w^\frac{p}{4}$, we have
\begin{align*}
	\frac{w^\frac{p}{2}}{|y|^2}\in L^1(B_{\frac{1}{2}}),
\end{align*}
and applying $(\ref{1.1})$ with $p=2, s=0$ to $\eta w^\frac{p}{4}$, we also have 
\begin{align*}
	w^\frac{np}{2(n-2)}\in L^1(B_{\frac{1}{2}}).
\end{align*}
Then if $w^\frac{\alpha+p}{2}, \frac{w^\frac{\alpha+1}{2}}{|y|^2}\in L^1(B_r)$, we have $w^\frac{(\alpha+p)n}{2(n-2)},\frac{w^\frac{\alpha+p}{2}}{|y|^2}\in L^1(B_{\sigma r})$, where  $r, \sigma<1$.

Let $\alpha_{i+1}=\min\{\frac{n\alpha_i}{n-2}+\frac{2p}{n-2},\alpha_i+p-1\}$, then
\begin{align}
	\alpha_{i+1}-\alpha_i\geq \min\left\{\frac{2p}{n-2},p-1\right\}
\end{align}
Then $\alpha_i\rightarrow \infty$ as $i\rightarrow\infty$, after scaling we have $w^\frac{q}{2}\in L^1(B_1)$ for any $q>1$.

By Fatou Lemma,

\begin{align*}
	\int_{B_1}|\nabla u|^q u^\gamma =\int _{B_1} \lim_{\varepsilon\rightarrow 0}(|\nabla u_\varepsilon|^2+\varepsilon)^\frac{q}{2}u^\gamma\leq \lim_{\varepsilon\rightarrow 0}\int_{B_1} w^\frac{q}{2}u^\gamma<C.
\end{align*}

A scaling argument completes our proof.

\end{proof}
From the second estimates for the $p$-Laplace equation in \cite{AKM18} and \cite{CM18}, we have the following regularity result.
\begin{lemma}\label{lem3.5}
	Let $u$ be a solution of (\ref{1.2}), then for $k\geq 3$, $|\nabla u|^{p-2}\nabla u\in W^{1,2}_{loc}(\mathbb{R}^n)$.
\end{lemma}
\begin{proof}
	For $k\geq 3$, we have $\frac{u^{p*(1)-1}}{|y|}\in L^2_{loc}(\mathbb{R}^n)$. Using  the standard result in \cite{AKM18} or \cite{CM18}, we have done.

\end{proof}

\begin{lemma}
	Let $u$ be a solution of (\ref{1.2}), and  $Z:=\{x\in \mathbb{R}^n:\nabla u(x) =0\}$ and $W_\varepsilon:=\{x\in \mathbb{R}^n:|y|>\varepsilon\}$, then $|Z\cap W_\varepsilon|=0$ and $u\in C^{\infty}( W_\varepsilon\backslash Z)$.
\end{lemma}
\begin{proof}
	We have $\frac{u^{p*(1)-1}}{|y|}\in L^\infty(W_\varepsilon)$, so from \cite{ACF}, $u\in C^{1,\beta}(W_\varepsilon)$ and $Z$ has zero measure since $\frac{u^{p*(1)-1}}{|y|}>0$. By a bootstrap  argument, $ u\in C^\infty(W_\varepsilon\backslash Z)$.

	\end{proof}

\section{Classification of cylindrically symmetric solutions}
In this section, we prove Theorem \ref{them1.4}. First, we derive an identity.

Let $u$ be a cylindrically symmetric solution of $(\ref{1.2})$, $r=|y|,t=|z|$ and we define
\begin{align*}
	u(y,z)=\kappa\phi(r,t)^{-\frac{n-p}{p}},
\end{align*}
where $\kappa$ is constant to be determined later. Then $\phi$ shares the regularities with $u$. In the coordinate $r, s$, then
\begin{equation}\label{4.0}
	\Delta_p u= (p-2)|\nabla u|^{p-4}[u_r^2 u_{rr} + 2 u_ru_su_{rs}+ u_s^2u_{ss}]\\
	+ |\nabla u|^{p-2}[u_{rr} + \frac{k-1}{r}u_r + u_{ss}^2 + \frac{n-k-1}{s}u_s].
\end{equation}

It follows that
\begin{equation}\label{4.1}
	\Delta_p\phi+\frac{k-1}{r}|\nabla \phi|^{p-2}\phi_r+\frac{n-k-1}{t}|\nabla \phi|^{p-2}\phi_t=\frac{n(p-1)}{p}\frac{|\nabla \phi|^{p}}{\phi}+\frac{a}{r}.
\end{equation}
where $a=\frac{\kappa^{\frac{p(p-1)}{n-p}}p^{p-1}}{(n-p)^{p-1}}$, and we will compute $\kappa$ by $a$. \\
We define $X^i=|\nabla \phi|^{p-2}\phi_i, X^i_j=(|\nabla\phi|^{p-2}\phi_i)_j$, $E^{i}_j=X^{i}_j-\frac{tr(X^{i}_j)}{2}\delta_{ij}$,\quad $i,j=r,t$,
\begin{align*}
h=r^{k-1}t^{n-k-1},h_r=\frac{k-1}{r}h,h_t=\frac{n-k-1}{t}h.
\end{align*}

For $\varepsilon>0$, we do the following calculations in $W_\varepsilon\backslash Z$ since $\phi$ is smooth in $W_\varepsilon\backslash Z$. Now, we begin to compute,
\begin{align*}
	(E^{i}_jX^{j}\phi^{1-n}h)_i=&E^{i}_jE^{j}_i\phi^{1-n}h+(E^{i}_{j})_i|\nabla \phi|^{p-2}\phi_j\phi^{1-n}h+\\&(1-n)E^{i}_j|\nabla \phi|^{p-2}\phi_j\phi_i\phi^{-n}h+E^i_j|\nabla \phi|^{p-2}\phi_j\phi^{1-n}h_i\\=&I+II+III+IV.
\end{align*}
To deal with $II$, we note that 
\begin{align*}
	(E^i_j)_i=\frac{1}{2}(\Delta_p\phi)_j.
\end{align*}
Let 
\begin{align*}
	\Delta_p\phi=\frac{n(p-1)}{p}\frac{|\nabla\phi|^p}{\phi}+f,
\end{align*}
where 
\begin{align*}
	f=\frac{a}{r}-\frac{k-1}{r}|\nabla \phi|^{p-2}\phi_r-\frac{n-k-1}{t}|\nabla \phi|^{p-2}\phi_t.
\end{align*}
Then 
\begin{align*}
	(\Delta_p\phi)_j|\nabla \phi|^{p-2}\phi_j\phi^{1-n}h=&n(p-1)|\nabla\phi|^{2p-4}\phi_i\phi_j\phi_{ij}\phi^{-n}h-\frac{n(p-1)}{p}|\nabla \phi|^{2p}\phi^{-1-n}h+\\&f_j|\nabla\phi|^{p-2}\phi_j\phi^{1-n}h.
\end{align*}
For the second term II,
\begin{align*}
	II=&n(p-1)|\nabla\phi|^{2p-4}\phi_i\phi_j\phi_{ij}\phi^{-n}h-\frac{n(p-1)}{p}|\nabla \phi|^{2p}\phi^{-1-n}h+\\&f_j|\nabla\phi|^{p-2}\phi_j\phi^{1-n}h-\frac{1}{2}(\Delta_p \phi)_j|\nabla \phi|^{p-2}\phi_j\phi^{1-n}h.
	\end{align*}
And for the third term III,
\begin{align*}
	III=&(1-n)(|\nabla \phi|^{p-2}\phi_i)_j\phi_i\phi_j|\nabla \phi|^{p-2}\phi^{-n}h-\frac{1-n}{2}\Delta_p\phi|\nabla\phi|^{p}\phi^{-n}h\\=&(|\nabla \phi|^{p-2}\phi_i)_j\phi_i\phi_j|\nabla\phi|^{p-2}\phi^{-n}h-n(p-1)|\nabla \phi|^{2p-4}\phi_i\phi_j\phi_{ij}\phi^{-n}h\\&-\frac{1-n}{2}\Delta_p\phi|\nabla\phi|^{p}\phi^{-n}h.
\end{align*}
Since
\begin{align*}
	(|\nabla\phi|^{p-2} \phi_i)_j\phi_i\phi_j &=(p-2)|\nabla \phi|^{p-4}\phi_l\phi_{lj}\phi_i\phi_i\phi_j+|\nabla \phi|^{p-2}\phi_{ij}\phi_i\phi_j\\&=(p-1)|\nabla \phi|^{p-2}\phi_{ij}\phi_i\phi_j,
\end{align*}
and 
\begin{align*}
	IV=(|\nabla \phi|^{p-2}\phi_i)_j|\nabla \phi|^{p-2}\phi_jh_i\phi^{1-n}-\frac{1}{2}\Delta_p\phi|\nabla \phi|^{p-2}\phi_ih_i\phi^{1-n}.
\end{align*}
It follows that we have 
\begin{align*}
	(E^{i}_jX^j\phi^{1-n}h)_i=&E^i_jE^j_i\phi^{1-n}h-\frac{n(p-1)}{p}|\nabla \phi|^{2p}\phi^{-1-n}h+f_j|\nabla\phi|^{p-2}\phi_j\phi^{1-n}h\\&-\frac{1}{2}(\Delta_p \phi)_j|\nabla \phi|^{p-2}\phi_j\phi^{1-n}h+(|\nabla \phi|^{p-2}\phi_i)_j\phi_i\phi_j|\nabla\phi|^{p-2}\phi^{-n}h\\&-\frac{1-n}{2}\Delta_p\phi|\nabla\phi|^{p}\phi^{-n}h\\&+(|\nabla \phi|^{p-2}\phi_i)_j|\nabla \phi|^{p-2}\phi_jh_i\phi^{1-n}-\frac{1}{2}\Delta_p\phi|\nabla \phi|^{p-2}\phi_ih_i\phi^{1-n}\\
	=&E^i_jE^j_i\phi^{1-n}h-\frac{n(p-1)}{p}|\nabla \phi|^{2p}\phi^{-1-n}h+f_j|\nabla\phi|^{p-2}\phi_j\phi^{1-n}h\\&+(|\nabla \phi|^{p-2}\phi_i)_j\phi_i\phi_j|\nabla\phi|^{p-2}\phi^{-n}h+(|\nabla \phi|^{p-2}\phi_i)_j|\nabla \phi|^{p-2}\phi_jh_i\phi^{1-n}\\&-\frac{1}{2}[\Delta_p\phi|\nabla \phi|^{p-2}\phi_i\phi^{1-n}h]_i+\frac{1}{2}(\Delta_p\phi)^{2}\phi^{1-n}h,	
	\end{align*}
since
\begin{align*}
	&-\frac{1}{2}(\Delta_p\phi)_i|\nabla \phi|^{p-2}\phi_i\phi^{1-n}h\\=&-\frac{1}{2}[(\Delta_p\phi)|\nabla\phi|^{p-2}\phi_i\phi^{1-n}h]_i+\frac{1}{2}(\Delta_p\phi)^2\phi^{1-n}h\\&+\frac{1-n}{2}\Delta_p\phi|\nabla \phi|^{p}\phi^{-n}h+\frac{1}{2}\Delta_p\phi|\nabla \phi|^{p-2}\phi_ih_i\phi^{1-n}.
\end{align*}
Observe that
\begin{align*}
&\left[\frac{p-1}{p}|\nabla \phi|^{p}(|\nabla \phi|^{p-2}\phi_i)\phi^{-n}h\right]_i\\=&\frac{p-1}{p}(p|\nabla \phi|^{2p-4}\phi_{ij}\phi_j\phi_i\phi^{-n}h+|\nabla \phi|^{p}\Delta_p\phi\phi^{-n}h\\&+|\nabla\phi|^{2p-2}\phi_{i}h_i\phi^{-n}-n|\nabla\phi|^{2p}\phi^{-1-n}h),
\end{align*}
then
\begin{align*}
	(|\nabla \phi|^{p-2}\phi_i)_j\phi_i\phi_j|\nabla\phi|^{p-2}\phi^{-n}h=&(p-1)|\nabla \phi|^{2p-4}\phi_i\phi_j\phi_{ij}\phi^{-n}h\\=&[\frac{p-1}{p}|\nabla \phi|^p(|\nabla \phi|^{p-2}\phi_i)\phi^{-n}h]_i-\frac{p-1}{p}|\nabla\phi|^p\Delta_p\phi\phi^{-n}h\\&-\frac{p-1}{p}|\nabla \phi|^{2p-2}h_i\phi_i\phi^{-n}+\frac{n(p-1)}{p}|\nabla\phi|^{2p}\phi^{-1-n}h.
\end{align*}
No we get
\begin{equation}\label{4.2}
\begin{aligned}
 &(E^i_jX^j\phi^{1-n}h)_i\\
 =&E^i_jE^j_i\phi^{1-n}h-\frac{1}{2}(\Delta_p\phi |\nabla\phi|^{p-2}\phi_i\phi^{1-n}h)_i+\frac{p-1}{p}[|\nabla\phi|^{p}(|\nabla\phi|^{p-2}\phi_i)\phi^{-n}h]_i\\
 &+\frac{1}{2}(\Delta_p\phi)^2\phi^{1-n}h-\frac{p-1}{p}\Delta_p\phi|\nabla\phi|^p\phi^{-n}h+A,
\end{aligned}
\end{equation}
where,
\begin{align*}
	A=&-\frac{p-1}{p}|\nabla\phi|^{2p-2}h_i\phi_i\phi^{-n}+(|\nabla \phi|^{p-2}\phi_i)_j|\nabla\phi|^{p-2}h_i\phi_j\phi^{1-n}+f_i\phi_i|\nabla\phi|^{p-2}\phi^{1-n}h\\
	=&-\frac{p-1}{p}\frac{k-1}{r}|\nabla\phi|^p|\nabla\phi|^{p-2}\phi_r\phi^{-n}h-\frac{p-1}{p}\frac{n-k-1}{t}|\nabla\phi|^p|\nabla\phi|^{p-2}\phi_t\phi^{-n}h\\
	&+\frac{k-1}{r^2}|\nabla\phi|^{2p-4}\phi_r^2\phi^{1-n}h-\frac{a}{r^2}|\nabla\phi|^{p-2}\phi_r\phi^{1-n}h+\frac{n-k-1}{t^2}|\nabla\phi|^{2p-4}\phi_t^2\phi^{1-n}h.
\end{align*}
If we multiply (\ref{4.1}) with $|\nabla\phi|^p {\phi}^{1-n}h$, then
\begin{equation}\label{4.2a}
	\begin{aligned}
&\frac{1}{2}(\Delta_p\phi)^2\phi^{1-n}h-\frac{p-1}{p}\Delta_p\phi|\nabla\phi|^p\phi^{-n}h+A\\=&\frac{1}{2}(\Delta_p\phi)^2\phi^{1-n}h-\frac{n(p-1)^2}{p^2}|\nabla\phi|^{2p}\phi^{-1-n}h-\frac{p-1}{p}\frac{a}{r}|\nabla\phi|^{p}\phi^{-n}h\\
&+\frac{k-1}{r^2}|\nabla\phi|^{2p-4}\phi_r^2\phi^{1-n}h-\frac{a}{r^2}|\nabla\phi|^{p-2}\phi_r\phi^{1-n}h+\frac{n-k-1}{t^2}|\nabla\phi|^{2p-4}\phi_t^2\phi^{1-n}h.
\end{aligned}
\end{equation}
To eliminate $\frac{a}{r}|\nabla\phi|^p\phi^{-n}h$, we introduce 
\begin{align}\label{4.2b}
	B=\left\{\frac{p-1}{p}\frac{|\nabla\phi|^p}{\phi}\phi^{1-n}h-(|\nabla\phi|^{p-2}\phi_r)_r\phi^{1-n}h\right\}_r-\{(|\nabla\phi|^{p-2}\phi_t)_r\phi^{1-n}h\}_t
\end{align}
Differentiating  $(\ref{4.1})$ with respect to $r$, we get 
\begin{equation}\label{4.2c}
	\begin{aligned}
	B&=\left(\frac{p-1}{p}\frac{|\nabla\phi|^p}{\phi}\right)_r\phi^{1-n}h+\frac{p-1}{p}\frac{k-1}{r}|\nabla\phi|^p\phi^{-n}h+\frac{p-1}{p}(1-n)\frac{|\nabla\phi|^p}{\phi}\frac{\phi_r}{\phi}\phi^{1-n}h\\&+\left\{-\frac{k-1}{r^2}|\nabla\phi|^{p-2}\phi_r-\frac{n(p-1)}{p}(\frac{|\nabla\phi|^p}{\phi})_r+\frac{a}{r^2}\right\}\phi^{1-n}h\\&-(1-n)\{(|\nabla\phi|^{p-2}\phi_r)_r\frac{\phi_r}{\phi}+(|\nabla\phi|^{p-2}\phi_t)_r\frac{\phi_t}{\phi}\}\phi^{1-n}h\\&=\frac{p-1}{p}\frac{k-1}{r}|\nabla\phi|^p\phi^{-n}h-\frac{k-1}{r^2}|\nabla\phi|^{p-2}\phi_r\phi^{1-n}h+\frac{a}{r^2}\phi^{1-n}h,
\end{aligned}
\end{equation}
if we show
\begin{align*}
	\frac{p-1}{p}\left(\frac{|\nabla\phi|^p}{\phi}\right)_r+\frac{p-1}{p}\frac{|\nabla\phi|^p}{\phi}\frac{\phi_r}{\phi}-(|\nabla\phi|^{p-2}\phi_r)_r\frac{\phi_r}{\phi}-(|\nabla\phi|^{p-2}\phi_t)_r\frac{\phi_t}{\phi}=0.
\end{align*}
Indeed, note that
\begin{align*}
	\left(\frac{|\nabla\phi|^p}{\phi}\right)_r=\frac{p|\nabla\phi|^{p-2}(\phi_r\phi_{rr}+\phi_t\phi_{tr})}{\phi}-\frac{|\nabla\phi|^{p}}{\phi}\frac{\phi_r}{\phi},
\end{align*}
it is equivalent to verify that	
\begin{align*}
	(p-1)|\nabla\phi|^{p-2}(\phi_r\phi_{rr}+\phi_t\phi_{tr})=(|\nabla\phi|^{p-2}\phi_r)_r\phi_r+(\nabla\phi|^{p-2}\phi_t)_r\phi_t,
\end{align*}
which is correct by  direct computation.\\

So we combine (\ref{4.2}) and (\ref{4.2c}),  $(\ref{4.2})+\frac{a}{k-1}B$ gives
\begin{equation}\label{4.3}
	\begin{aligned}
		&(E^i_jX^j\phi^{1-n}h)_i+\frac{1}{2}[\Delta_p\phi|\nabla\phi|^{p-2}\phi_i\phi^{1-n}h]_i-\frac{p-1}{p}(|\nabla\phi|^p|\nabla\phi|^{p-2}\phi_i\phi^{-n}h)_i\\&+\frac{a}{k-1}\left\{\bigg[\frac{p-1}{p}\frac{|\nabla\phi|^p}{\phi}\phi^{1-n}h-(|\nabla\phi|^{p-2}\phi_r)_r\phi^{1-n}h\bigg]_r-[(|\nabla\phi|^{p-2}\phi_t)_r\phi^{1-n}h]_t\right\}\\&=E^{i}_jE^j_i\phi^{1-n}h+\frac{n}{2(n-2)}\left(\Delta_p\phi-\frac{2(p-1)}{p}\frac{|\nabla\phi|^p}{\phi}\right)^2\phi^{1-n}h\\&+\frac{(k-1)(n-k-1)}{n-2}\left[\frac{|\nabla\phi|^{p-2}\phi_t}{t}-\frac{|\nabla\phi|^{p-2}\phi_r}{r}+\frac{a}{(k-1)r}\right]^2\phi^{1-n}h,
	\end{aligned}	
\end{equation}
by
\begin{align*}
	&\frac{1}{n-2}\left(\Delta_p\phi-\frac{n(p-1)}{p}\frac{|\nabla\phi|^p}{\phi}\right)^2+\frac{1}{2}(\Delta_p\phi)^2-\frac{n(p-1)^2}{p^2}\frac{|\nabla\phi|^{2p}}{\phi^2}\\=&\frac{n}{2(n-2)}\left(\Delta_p\phi-\frac{2(p-1)}{p}\frac{|\nabla\phi|^p}{\phi}\right)^2,
\end{align*}
and 
\begin{align*}
	&-\frac{1}{n-2}\left[(n-k-1)\frac{|\nabla\phi|^{p-2}\phi_t}{t}+(k-1)\frac{|\nabla\phi|^{p-2}\phi_r}{r}-\frac{a}{r}\right]^2\\&+\frac{k-1}{r^2}|\nabla\phi|^{2p-4}\phi_r^2-\frac{2a}{r^2}|\nabla\phi|^{p-2}\phi_r+\frac{n-k-1}{t^2}|\nabla\phi|^{2p-4}\phi_t^2+\frac{a^2}{(k-1)r^2}\\&=\frac{(k-1)(n-k-1)}{(n-2)}\left[\frac{|\nabla\phi|^{p-2}\phi_t}{t}-\frac{|\nabla\phi|^{p-2}\phi_r}{r}+\frac{a}{(k-1)r}\right]^2.
\end{align*}
Therefore, if we define
\begin{align}\label{4.3a}
	F=X^{r}_{r}X^r+X^r_tX^t-\frac{p-1}{p}\frac{|\nabla\phi|^p}{\phi}X^r+\frac{a}{k-1}\left(\frac{p-1}{p}\frac{|\nabla\phi|^{p}}{\phi}-X^r_r\right),
\end{align} 
and 
\begin{align}\label{4.3b}
	G=X^t_rX^r+X_t^tX^t-\frac{p-1}{p}\frac{|\nabla\phi|^{p}}{\phi}X^t-\frac{a}{k-1}X^t_r,
\end{align}
and let $Y$ be vector field
\begin{align}\label{4.3c}
	Y=(F\phi^{1-n}h,G\phi^{1-n}h)=\left(X^i_j-\frac{p-1}{p}\frac{|\nabla\phi|^p}{\phi}\delta_{ij}\right)\left(X^j-\frac{a}{k-1}\delta_{rj}\right){\phi}^{1-n}h.
\end{align}
We have the following key differential identity.
\begin{proposition}\label{prop4.1} For any $\varepsilon>0$, the following identity holds in $W_\varepsilon\backslash Z$.
\begin{equation}\label{4.4}
	\begin{aligned}
		divY=&E^{i}_jE^j_i\phi^{1-n}h+\frac{n}{2(n-2)}\left(\Delta_p\phi-\frac{2(p-1)}{p}\frac{|\nabla\phi|^p}{\phi}\right)^2\phi^{1-n}h\\&+\frac{(k-1)(n-k-1)}{n-2}\left[\frac{|\nabla\phi|^{p-2}\phi_t}{t}-\frac{|\nabla\phi|^{p-2}\phi_r}{r}+\frac{a}{(k-1)r}\right]^2\phi^{1-n}h\\
	=&\left(X^i_j-\frac{p-1}{p}\frac{|\nabla \phi|^p}{\phi}\delta_{ij}\right)\left(X^j_i-\frac{p-1}{p}\frac{|\nabla \phi|^p}{\phi}\delta_{ij}\right)\phi^{1-n}h\\
	&+\frac{n}{2(n-2)}\left(\Delta_p\phi-\frac{2(p-1)}{p}\frac{|\nabla\phi|^p}{\phi}\right)^2\phi^{1-n}h\\
	&+\frac{(k-1)(n-k-1)}{n-2}\left[\frac{|\nabla\phi|^{p-2}\phi_t}{t}-\frac{|\nabla\phi|^{p-2}\phi_r}{r}+\frac{a}{(k-1)r}\right]^2\phi^{1-n}h	\geq 0		\end{aligned}
	\end{equation}
\end{proposition}
\begin{proof}
From (\ref{4.3}), we get the  differential identity in (\ref{4.4}), in order to prove the right of (\ref{4.4}) is nonnegative,	it suffices to show 
	\begin{align*}
		\left(X^i_j-\frac{p-1}{p}\frac{|\nabla \phi|^p}{\phi}\delta_{ij}\right)\left(X^j_i-\frac{p-1}{p}\frac{|\nabla \phi|^p}{\phi}\delta_{ij}\right)\geq c\left(X^i_j-\frac{p-1}{p}\frac{|\nabla \phi|^p}{\phi}\delta_{ij}\right)^2,
		\end{align*}
		for some $c>0$. From the Lemma 4.5 in  \cite{CM18}, we only need  to show that
		\begin{align}\label{4.4a}
			X^i_j-\frac{p-1}{p}\frac{|\nabla \phi|^p}{\phi}\delta_{ij}=(N_1N_2)_{ij},
		\end{align} 
		where $N_1$ is positive define and $N_2$ is symmetric. 
		
		First,
		\begin{align*}
			X^i_j-\frac{p-1}{p}\frac{|\nabla\phi|^{p}}{\phi}\delta_{ij}=N_3N_4+N_5,
		\end{align*}
		where $N_4=|\nabla \phi|^{p-2}\phi_{ij}$, $N_3=(p-2)\frac{\phi_i\phi_j}{|\nabla\phi|^2}+\delta_{ij}$ and  $N_5=\frac{p-1}{p}\frac{|\nabla \phi|^{p}}{\phi}\delta_{ij}$. Then $N_3$ is positive  define with eigenvalues $1$ and $p-1$. $N^{-1}_3=\delta_{ij}-\frac{p-2}{p-1}\frac{\phi_i\phi_j}{|\nabla\phi|^2}$.
		
		It follows that
		\begin{align*}
			X^i_j-\frac{p-1}{p}\frac{|\nabla\phi|^{p}}{\phi}\delta_{ij}=N_3(N_4+N_3^{-1}N_5)
		\end{align*}
		where $N_4+N_{3}^{-1}N_5$ is symmetric.
		
		Setting $N_1=N_3$, $N_2=N_4+N_3^{-1}N_5$, we have done. 		
						\end{proof}

	In order to get the integral estimates, we need the following estimates.				
\begin{lemma}\label{lemma4.2}
	For any $\epsilon >0$, we have 
\begin{equation}
	\begin{aligned}
	  \left(X^i_j-\frac{p-1}{p}\frac{|\nabla \phi|^p}{\phi}\delta_{ij}\right)\left(X^j-\frac{a}{k-1}\delta_{rj}\right)\zeta_j\leq &\epsilon \left(X^i_j-\frac{p-1}{p}\frac{|\nabla \phi|^p}{\phi}\delta_{ij}\right)\left(X^j_i-\frac{p-1}{p}\frac{|\nabla \phi|^p}{\phi}\delta_{ij}\right)\\&+C(\epsilon)|\zeta_j|^2(|X^i|^2+1).	
	    	\end{aligned}\end{equation}
	
\end{lemma}
\begin{proof}
	The same notations as above, from (\ref{4.4a}), we have
	\begin{align*}
		X^i_j-\frac{p-1}{p}\frac{|\nabla\phi|^p}{\phi}\delta_{ij}=AC^TBC,
			\end{align*}
			where $C$ is a orthogonal matrix, 
			\begin{align*}
				C^{-1}=C^T,
			\end{align*}
			
			$A$ is a symmetric matrix,   $(B)_{ij}=(p-2)\frac{\phi_i\phi_j}{|\nabla\phi|^2}+\delta_{ij}$ where $B$ is positive  define with eigenvalues $1$ and $p-1$, and  $(BC)_{ij}=\lambda_i\delta_{ij}, \lambda_1=p-1, \lambda_{i}=1, i=2,...,n$.
			It follows that
			\begin{align*}
				&tr(AC^TBCD)\\=&(AC^T)_{ij}(BC)_{jk}D_{ki}\\=&(AC^T)_{ij}\lambda_j D_{ji}\\\leq &\epsilon (AC^T)^2_{ij}+C(\epsilon) D_{ij}^2\\\leq &\epsilon tr (AC^T(AC^{T})^T)+C(\epsilon) tr(D^2)
			\end{align*}
			and 
			\begin{align*}
				tr(AC^TBCAC^{T}BC)=(AC^T)_{ij}\lambda_j(AC^T)_{ji}\lambda_i\geq \lambda^2 tr(AC^T(AC^T)^T)
			\end{align*}
		where $\lambda=\min\{1, p-1\}$.
			Then we have
			\begin{align}\label{4.4b}
				tr(AC^TBCD)\leq \epsilon tr(AC^TBCAC^TBC)+C(\epsilon)tr(D^2)
			\end{align}
			Substituting $(D)_{ij}=(X^j-\frac{a}{k-1}\delta_{rj})\zeta_i$ and using
			\begin{align*}
				tr(D^2)=(X^i-\frac{a}{k-1}\delta_{ir})\zeta_j (X^j-\frac{a}{k-1}\delta_{jr})\zeta_i\leq C(|X^i|^2+1)|\zeta_j|^2,
			\end{align*}
			we have completed the proof.
			\end{proof}
Now we begin the proof of the main Theorem \ref{them1.4}, as in the paper by Mancini-Fabbri-Sandeep \cite{MFS06}, we shall use the key differential identity (\ref{4.4}) ,  the $C^0$ estimates (\ref{3.1}) , the integral gradient estimates in (\ref{3.3}) and (\ref{3.4c}).

\textbf{Proof of Theorem \ref{them1.4}}.   Let $R>0$ and define
\begin{align*}
	\Omega_{R}=\{(r,t):0<r^2+t^2<4R^2\}.
\end{align*}
Let $\zeta(r,t)$ be cut-off function in $\mathbb{R}^2$ with $\zeta=1$ in $(r^2+t^2)^{\frac{1}{2}}\leq R$, $\zeta=0$ in $(r^2+t^2)^{\frac{1}{2}}\geq 2R$ and $|\nabla \zeta|\leq \frac{C}{R}$. 
\begin{align*}
	T_\varepsilon=\{(r,t): 0<r<\varepsilon\},
\end{align*}
Let $\varphi(r,t)$ be a smooth function in $\mathbb{R}^2$ with $\varphi=0$ in $T_\varepsilon$, $\varphi=1$ outside $T_{2\varepsilon}$ and $\varphi_t=0$, $|\varphi_r|\leq \frac{C}{\varepsilon}$.

Since $Z$ has zero measure, for the vector field $Y$ in (\ref{4.3c}), integration by parts gives
\begin{align}\label{4.6}
	\int_{\Omega_{R}}\varphi^l\zeta^m divY=\int_{\partial\Omega_{R}}\varphi^l\zeta^m Y\cdot \nu dH^1-\int_{\Omega_{R}}\nabla(\varphi^l \zeta^m) \cdot Y,
\end{align}
where $m, l >2$ are big, and $\nu$ is the outward normal to $\partial\Omega_{R}$ and $dH^1$ is the surface measure on the boundary.
\begin{equation}\label{4.7}
	\begin{aligned}
\int_{\partial\Omega_{R}}\varphi^l\zeta^m Y\cdot \nu dH^1=&-\int_{\Gamma_1}\varphi^l\zeta^m G\phi^{1-n}hdH^1-\int_{\Gamma_2}\varphi^l\zeta ^mF\phi^{1-n}hdH^1\\&+\int_{r^2+t^2=4{R}^2} \varphi^l\zeta^m Y\cdot \nu dH^1,
\end{aligned}\end{equation}

where $\Gamma_1=\partial\Omega_{R}\cap\{t=0\}$, $\Gamma_2=\partial\Omega_{R}\cap\{r=0\}$. 

From the symmetry of $\phi$, we have $\phi_t(r,0)=0$. We observe that $G=0$ on $\Gamma_1$ since $|\nabla\phi|^{p-2}\phi_t=(|\nabla\phi|^{p-2}\phi_t)_r=0$ on $\Gamma_1$, $\varphi=0$ on $\Gamma_2$ and $\zeta=0$ on $r^2+t^2=4R^2$, therefore 
\begin{align}\label{4.8}
	\int_{\partial\Omega_R}\varphi^l\zeta^m Y\cdot \nu dH^1=0.
\end{align}

Now we need to deal with 
\begin{align}\label{4.9}
	-\int_{\Omega_R}\nabla(\varphi^l\zeta^m)\cdot Y=-l\int_{\Omega_R} \varphi^{l-1}\zeta^m\nabla\varphi\cdot Y -m\int_{\Omega_R}\varphi^l\zeta^{m-1}\nabla\zeta\cdot Y.
\end{align}

By the definition of $Y=\left(X^i_j-\frac{p-1}{p}\frac{|\nabla\phi|^p}{\phi}\delta_{ij}\right)\left(X^j-\frac{a}{k-1}\delta_{rj}\right)\phi^{1-n}h
$ and using Lemma \ref{lemma4.2} we have, 
\begin{equation}\label{4.10}
\begin{aligned}
	m\varphi^l\zeta^{m-1} |\nabla \zeta \cdot Y|\leq& 0.1\left(X^i_j-\frac{p-1}{p}\frac{|\nabla \phi|^p}{\phi}\delta_{ij}\right)\left(X^j_i-\frac{p-1}{p}\frac{|\nabla \phi|^p}{\phi}\delta_{ij}\right)\varphi^l\zeta^{m} \phi^{1-n}h\\&+\frac{C}{R^2}(|\nabla \phi|^{p-1}+1)^2\phi^{1-n}h\varphi^l \zeta^{m-2},
\end{aligned}
\end{equation}
in the support set of $\nabla\zeta\subset \Omega_{R}\backslash\Omega_\frac{R}{2}$.

 Using Lemma \ref{lem3.3} and Lemma \ref{lem3.4} (with $q=2p)$ we have
 \begin{align}\label{4.11}
 	\int_{B_R}|\nabla u|^{2p-2}u^\gamma \leq \left(\int_{B_R}|\nabla u|^{2p}u^\gamma\right)^\frac{p-1}{p}\left(\int_{B_R} u^\gamma\right)^\frac{1}{p}\leq C(1+R^{n-2(n-1)-\frac{\gamma(n-p)}{p-1}}).
 \end{align}
 Then
 \begin{align}\label{4.12}
 	\int_{\Omega_{R}}|\nabla\phi|^{2p-2}\phi^{1-n}r^{k-1}t^{n-k-1}=C\int_{B_{2R}}|\nabla u|^{2p-2}u^{\frac{p(n-1)}{n-p}-\frac{(2p-2)n}{n-p}} \leq C(1+R^{2-\frac{n-p}{p-1}}),
 	\end{align}
 Besides, by Proposition \ref{pro3.1} we have $c_1(1+r^2+t^2)^\frac{p}{2(p-1)}\leq \phi \leq c_2(1+r^2+t^2)^\frac{p}{2(p-1)}$, then
  \begin{align}\label{4.13}
 	\int_{\Omega_R}\phi^{1-n}r^{k-1}t^{n-k-1}\leq C(1+R^{-\frac{n-p}{p-1}}).
 \end{align}
 So by (\ref{4.12}) and (\ref{4.13}), we get
\begin{align}\label{4.14}
	\frac{1}{R^2}\int_{\Omega_R\backslash\Omega_{\frac{R}{2}}}\varphi^l\zeta^{m-2}(|\nabla \phi|^{p-1}+C)^2\phi^{1-n}h\leq C R^{-\tau},\end{align}
where $\tau=\min\{\frac{n-p}{p-1},2\}$.

Let
\begin{align}\label{4.15}
	Q_\varepsilon= {\Omega_R \cap \{T_{2\varepsilon}\backslash T_{\varepsilon}\}}\subset \Omega_R
\end{align}
Next, we estimate 
\begin{align}\label{4.16}
	\int_{\Omega_R}\varphi^{l-1}\zeta^{m}\nabla\zeta\cdot Y=
\int_{Q_\varepsilon} \varphi^{l-1} \zeta^m\varphi_r F \phi^{1-n}r^{k-1}t^{n-k-1}.
\end{align} 

For $k\geq 3$, the same as (\ref{4.10}),
\begin{equation}\label{4.17}
\begin{aligned}
	\varphi^{l-1} \zeta^m |\varphi_rF|\phi^{1-n}h\leq & 0.01\left(X^i_j-\frac{p-1}{p}\frac{|\nabla \phi|^p}{\phi}\delta_{ij}\right)\left(X^j_i-\frac{p-1}{p}\frac{|\nabla \phi|^p}{\phi}\delta_{ij}\right)\varphi^l\zeta^m\phi^{1-n}h\\&+\frac{C}{\varepsilon ^2}(|\nabla \phi|^{p-1}+1)^2\varphi^{l-2}\zeta^m\phi^{1-n}h,
	\end{aligned}
	\end{equation}
	in the support set of $\varphi_r\subset Q_\varepsilon$.
	
We have
\begin{equation}\label{4.18}
	\begin{aligned}
	&\frac{1}{\varepsilon ^2}\int_{Q_\varepsilon}\phi^{1-n}h\leq \frac{C}{\varepsilon^2}\int_\varepsilon ^{2\varepsilon}\int_{0}^{2R}(1+r^2+t^2)^{\frac{p(1-n)}{2(p-1)}}t^{n-k-1}dtr^{k-1}dr
	\\\leq &\frac{C}{\varepsilon^2}\int_\varepsilon ^{2\varepsilon}\int_0^{2R}(1+t^2)^{-\frac{(n-1)p}{2(p-1)}}t^{n-k-1}dt r^{k-1}dr\leq \frac{C}{\varepsilon^2}\int_\varepsilon ^{2\varepsilon} r^{k-1}dr
	\leq  C\varepsilon ^{k-2}.
\end{aligned}\end{equation}

Here we use $-\frac{(n-1)p}{p-1}+n-k-1=-\frac{n-1}{p-1}-k$, so $\int_{0}^\infty(1+t^2)^{-\frac{(n-1)p}{2(p-1)}}t^{n-k-1}dt\leq C$.
And by Holder's inequality, we have 
\begin{align}\label{4.19}
	\frac{1}{\varepsilon ^2}\int_{Q_\varepsilon}|\nabla \phi|^{2p-2}\phi^{1-n}h&\leq \frac{C}{\varepsilon ^2}\left(\int_{Q_\varepsilon}|\nabla \phi|^q\phi^{1-n}h\right)^{\frac{2p-2}{q}}\left(\int_{Q_\varepsilon}\phi^{1-n}h\right)^{\frac{q-2p+2}{q}}.
		\end{align}
		
	For $k>2p$, we set $q=2p$, if $k\leq 2p$, we set $q=\frac{2(p-1)k}{k-2.1}$,  by Lemma \ref{lem3.3} and Lemma \ref{lem3.4} it follows that	
	
	\begin{align}\label{4.20}
		\int_{Q_\varepsilon}|\nabla \phi|^q\phi^{1-n}h\leq\int_{\Omega_{R}}|\nabla\phi|^q\phi^{1-n}r^{k-1}t^{n-k-1}=C\int_{B_{2R}}|\nabla u|^{q}u^{\frac{p(n-1)}{n-p}-\frac{qn}{n-p}} dx\leq C(R).
			\end{align}	
Thus, by (\ref{4.19}) and (\ref{4.18}),
\begin{align}\label{4.21}
	\frac{1}{\varepsilon ^2}\int_{Q_\varepsilon}|\nabla \phi|^{2p-2}\phi^{1-n}h\leq \frac{C(R)}{\varepsilon^2}\left(\int_{Q_\varepsilon}\phi^{1-n}h\right)^{\frac{q-2p+2}{q}}\leq \frac{C(R)}{\varepsilon^2}\varepsilon^{\frac{k(q-2p+2)}{q}}=C(R)\varepsilon^\mu,
	\end{align}
	where
	\begin{align}\label{4.22}
		\mu=\frac{k(q-2p+2)}{q}-2.
	\end{align}
	
For $k>2p$, $q=2p$, then $\mu=\frac{k}{p}-2>0$. For $k\leq 2p$, $q=\frac{2(p-1)k}{k-2.1}$, then $$\mu=\frac{k(\frac{2pk-2k}{k-2.1}-\frac{2pk-4.2p}{k-2.1}+\frac{2k-4.2}{k-2.1})}{\frac{2(p-1)k}{k-2.1}}-2=\frac{\frac{4.2(p-1)k}{k-2.1}}{\frac{2(p-1)k}{k-2.1}}-2=2.1-2=0.1,$$	
	 we have $\mu>0$ in both cases.
	 
	 Combing (\ref{4.21}) and (\ref{4.18}),
\begin{align}\label{4.23}
	\frac{C}{\varepsilon ^2}\int_{Q_\varepsilon}(|\nabla \phi|^{p-1}+1)^2\phi^{1-n}h\leq C(R)\varepsilon^\mu,
	\end{align}
	
Using the fact (\ref{4.15}), then we arrive at
\begin{equation}\label{4.24}
	\begin{aligned}
		&\int_{\Omega_R}\varphi^l\zeta^m divY\\
	=&\int_{\Omega_R}\varphi^l\zeta^m\left(X^i_j-\frac{p-1}{p}\frac{|\nabla \phi|^p}{\phi}\delta_{ij}\right)\left(X^j_i-\frac{p-1}{p}\frac{|\nabla \phi|^p}{\phi}\delta_{ij}\right)\phi^{1-n}h\\&+\frac{1}{n-2}\varphi^l\zeta^m\left(\Delta_p\phi-\frac{2(p-1)}{p}\frac{|\nabla\phi|^p}{\phi}\right)^2\phi^{1-n}h\\&+\frac{(k-1)(n-k-1)}{n-2}\varphi^l\zeta^m\left[\frac{|\nabla\phi|^{p-2}\phi_t}{t}-\frac{|\nabla\phi|^{p-2}\phi_r}{r}+\frac{a}{(k-1)r}\right]^2\phi^{1-n}h	\\=&-\int _{\Omega_R}\nabla(\varphi^l\zeta^m)\cdot Y=-m\int_{\Omega_R}\varphi^l\zeta^{m-1}\nabla\zeta\cdot Y-l\int_{Q_\varepsilon} \varphi^{l-1}\zeta^m\varphi_r F
	\\\leq& 0.11\int_{\Omega_R}\left(X^i_j-\frac{p-1}{p}\frac{|\nabla \phi|^p}{\phi}\delta_{ij}\right)\left(X^j_i-\frac{p-1}{p}\frac{|\nabla \phi|^p}{\phi}\delta_{ij}\right)\varphi^l\zeta^{m} \phi^{1-n}h\\&+\frac{C}{R^2}\int_{\Omega_{R}}(|\nabla \phi|^{p-1}+1)^2\phi^{1-n}h\varphi^l \zeta^{m-2}+\frac{C}{\varepsilon^2}\int_{Q_\varepsilon}(|\nabla \phi|^{p-1}+1)^2\phi^{1-n}h\varphi^{l-2} \zeta^{m}.		\end{aligned}
	\end{equation}
	
From (\ref{4.14}), (\ref{4.23}), we have
\begin{align}\label{4.25}
	\frac{1}{2}\int_{\Omega_{R}}\varphi^l\zeta^mdiv Y\leq C R^{-\tau}+C(R)\varepsilon^\mu.
\end{align}
For any $R>0$, let $\varepsilon\rightarrow 0$, we have
\begin{align}\label{4.26}
	\frac{1}{2}\int_{\Omega_{R}}\zeta^mdiv Y\leq C R^{-\tau},
	\end{align}
then let $R\rightarrow \infty$,
\begin{align}\label{4.27}
	\frac{1}{2}\int_{r>0,t>0}divY=0.
\end{align}

Therefore by Proposition \ref{prop4.1},

\begin{align}\label{4.28}
	divY=0,
\end{align}
which means
\begin{align}\label{4.29}
	E^{i}_{j}=0,
	\end{align}
	\begin{align}\label{4.30}
	\Delta_p\phi=\frac{2(p-1)}{p}\frac{|\nabla\phi|^p}{\phi}.
\end{align}

Now we focus on $Z^c$ since $\phi \in C^{\infty}(Z^c\cap\{(r,t): r>0,t>0\})$.
From (\ref{4.29}), we have
\begin{align}\label{4.31}
	(|\nabla\phi|^{p-2}\phi_r)_t=0,(|\nabla\phi|^{p-2}\phi_t)_r=0,
\end{align}
and
\begin{align}
	(|\nabla \phi|^{p-2}\phi_r)_r=(|\nabla \phi|^{p-2}\phi_t)_t=\frac{1}{2}\Delta_p\phi.
\end{align}
If we differentiate $(\ref{4.31})$, then there is a constant $b$ such that
\begin{align*}
	(|\nabla\phi|^{p-2}\phi_r)_r=(|\nabla\phi|^{p-2}\phi_t)_t:=b,
\end{align*}
then
\begin{align*}
	|\nabla\phi|^{p-2}\phi_r=br+d_1, |\nabla\phi|^{p-2}\phi_t=bt+d_2,
\end{align*}
for some constants $d_1,d_2$.\\
By $(\ref{4.30})$ and $(\ref{4.1})$, 
\begin{align*}
	d_1=\frac{a}{k-1},d_2=0,
\end{align*}
which is
\begin{align}
	|\nabla\phi|^{p-2}\phi_r=br+\frac{a}{k-1}, |\nabla\phi|^{p-2}\phi_t=bt.
\end{align}
By $(\ref{4.30})$
\begin{align*}
	\phi=\frac{p-1}{pb}|\nabla\phi|^{p}.
\end{align*}
If we take $a=(k-1)(\frac{p}{p-1})^{p-1}$, then
\begin{align*}
	\phi&=\frac{p-1}{pb}[(bt)^2+(br+(\frac{p}{p-1})^{p-1})^2)]^\frac{p}{2(p-1)}\\&=\lambda^{-1}[(\lambda t)^2+(\lambda r+1)^2]^{\frac{p}{2(p-1)}},
\end{align*}
where $\lambda=b(\frac{p-1}{p})^{p-1}$. And
\begin{align*}
	\kappa=(k-1)^{\frac{n-p}{p(p-1)}}(n-p)^{\frac{n-p}{p}}(p-1)^{-\frac{n-p}{p}},
\end{align*}
then 
\begin{align*}
	u(y,z)=\lambda^{\frac{n-p}{p}}u_0(\lambda y,\lambda z),
\end{align*}
where $u_0$ is as in Theorem\ref{them1.4} in $Z^c$. Since $u$ is continuous and $Z$ has zero measure, then $ u(y,z)=\lambda^{\frac{n-p}{p}}u_0(\lambda y,\lambda z)  $ in $\mathbb{R}^n$.

\vspace{0.6cm}
{\it Acknowledgement.}
The authors was supported by  National Natural Science Foundation of China (grants 12141105) and National Key Research and Development Project (grants SQ2020YFA070080).

\end{document}